\documentclass[oneside, 10pt]{amsart}

\usepackage{amsmath, amsfonts, amssymb, amsthm, fourier}

\newtheorem{theorem}{Theorem}[section]
\newtheorem{lemma}[theorem]{Lemma}
\newtheorem{corollary}[theorem]{Corollary}

\theoremstyle{definition}

\newtheorem{remark}[theorem]{Remark}

\numberwithin{equation}{section}

\begin{document}

\title{Semi--norms of the Bergman projection}

\author{Marijan Markovi\'{c}}

\address{
Faculty of Natural Sciences and Mathematics\endgraf
University of Montenegro\endgraf
Cetinjski put b.b.\endgraf
81000 Podgorica\endgraf
Montenegro}

\email{marijanmmarkovic@gmail.com}

\subjclass[2010]{Primary 45P05, Secondary 47B38, 30H30}

\keywords{the  Bergman projection, the  Bloch space}

\begin{abstract}
It is known that the Bergman projection operator maps the space of essentially bounded
functions in the unit ball in the $d$-dimensional complex vector space onto the  Bloch
space.  This paper deals with the various  semi--norms    of   the  Bergman projection.
We improve  some recent results.
\end{abstract}

\maketitle

\section{Introduction and the main theorem}
\subsection{Introduction} First we introduce the basic notation we will use. Throughout
the paper the letter d will denote a fixed positive integer.  Let
$\left<z,w\right>$  stand    for the inner product in the complex      $d$-dimensional
space $\mathbf{C}^d$  given by
\begin{equation*}
\left<z,w\right> = z_1\overline{w}_1+\dots+z_d\overline{w}_d,
\end{equation*}
where  $z=(z_1,\dots,z_d)$ and $w=(w_1,\dots,w_d)$ are coordinate representations   of
$z$ and  $w$  in the standard base $\{e_1,\dots,e_d\}$ of $\mathbf{C}^d$. Norm      in
$\mathbf{C}^d$   induced by the inner   product  is denoted                  by $|z| =
\sqrt{\left<z,z\right>}$.   Denote   by $B$ the unit ball $\{z\in \mathbf C^d:|z|<1\}$.
We write  $dv$ for the  Lebesgue measure in $\mathbf{C}^d$      normalized on the unit
ball.

Following the notation from the Rudin monograph~\cite{RUDIN.BOOK.BALL} as well as from
the Forelli and  Rudin work~\cite{FORELLI.INDIANA}, associate with each complex number
$s=  \sigma+it,\, \sigma>-1$                                       the integral kernel
\begin{equation*}
K_s(z,w)   =   \frac{\left(1-|w|^2\right)^s}{\left(1-\left<z,w\right>\right)^{d+1+s}},
\end{equation*}
and let
\begin{equation*}
T_s f(z) =  c_s   \int_B  K_s(z,w)\, f(w)\, dv(w),\quad z\in B.
\end{equation*}
Here   it is assumed that the complex power evaluates to its principal branch and that
the integral exists.        The coefficient $c_s$ is chosen in a such way that for the
weighted  measure
\begin{equation*}
dv_s(w)=c_s(1-|w|^2)^s dv(w)
\end{equation*}
we have $v_s(B)=1$ (so that $T_s  1=1$). One can show   that
\begin{equation*}
c_s     =   \frac{\Gamma(d+s+1)}{\Gamma(s+1) \Gamma(d+1)},
\end{equation*}
where $\Gamma$ stands for the Gamma function. The operator     $T_s$   is the Bergman
projection  operator. For properties            of  the Bergman  projections we refer
to~\cite{RUDIN.BOOK.BALL, ZHU.BOOK.SPACES}.

Let $L^p(B)\, (1\le  p < \infty)$       stand for the Lebesgue space of all measurable
functions     in the unit ball of $\mathbf{C}^d$ whose modulus to the exponent $p$ is
integrable. For $p = \infty$       let it denote the space of all essentially bounded
measurable functions. Denote by      $\|\cdot\|_p$  the usual norm on $L^p(B)\, (1\le
p\le\infty)$.

Forelli and Rudin~\cite{FORELLI.INDIANA} proved that the operator $T_s$ maps $L^p(B)$
onto the Bergman space in the unit ball $L^p_a(B)$        continuously if and only if
$\sigma > 1/p -1 \, (1\le p<\infty)$.                       Moreover, they calculated
\begin{equation*}
\|T_\sigma\|_{ L^1(B)\rightarrow L^1_a(B)}  =
\frac{\Gamma(d+\sigma+1)}{\Gamma^2((d+\sigma+1)/2)}
\frac{\Gamma(\sigma)}{\Gamma(\sigma+1)},\quad \sigma>0,
\end{equation*}
and
\begin{equation*}
\|T_\sigma\|_{ L^2(B)\rightarrow L^2_a(B)} =
 \frac{\sqrt{\Gamma(2\sigma + 1)} }{\Gamma(\sigma+1)}, \quad \sigma>-1/2.
\end{equation*}
Very  recently  Liu~\cite{LIU.JFA}   established the     following two     estimates
\begin{equation*}
\frac {\Gamma((d+1)/p)\Gamma ((d+1)/q)}{\Gamma^2((d+1)/2)}\le
\|T_0\|_{L^p(B)\rightarrow L^p_a(B)}\le
\frac{\Gamma(d+1)} {\Gamma^2((d+1)/2)} \frac{\pi}{\sin(\pi/p)},
\end{equation*}
where $q=p/(p-1)\, (1<p<\infty)$. He also made the conjecture    that equality holds
in the left hand inequality. It would be of some interest     to extend  this result
to   the weighted case, i.e., for  any $\sigma>1/p-1$.

On the other hand,  it is known   that the operator $T_\sigma$ projects $L^\infty(B)$
continuously onto the  Bloch space   $\mathcal{B}$ in the unit ball in $\mathbf{C}^d$
(for every $\sigma>-1$).  This can be seen from Theorem 3.4 in~\cite{ZHU.BOOK.SPACES}
and the Choe paper~\cite{CHOE.PAMS}.     Recall   that the  Bloch space $\mathcal{B}$
contains  all  functions $f(z)$ analytic                            in $B$ for  which
\begin{equation*}
\sup_{z\in  B}\left(1-|z|^2\right)|\nabla f(z)|
\end{equation*}
is  finite. Here
\begin{equation*}
\nabla f(z)
= \left(\frac {\partial f(z)}{\partial z_1},\dots,\frac{\partial f(z)}{\partial z_n}\right)
\end{equation*}
is the complex  gradient of $f(z)$ at the point $z$. For the facts  concerning    the
Bloch   space in several dimensions we refer to~\cite{TIMONEY.BLMS, TIMONEY.CRELLE, ZHU.BOOK.SPACES}.

\subsection{The main result}           Let $n$ be  any   positive  integer and denote
$\tilde {d} = \binom {n+d-1}{d-1}$. In                               what follows let
$|\cdot|_{ {\mathbf{C}}^{\tilde{d}}}$ be a norm on $\mathbf{C}^{\tilde{d}}$     which
satisfies
\begin{equation*}
|\overline{Z}|_{ {\mathbf{C}}^{\tilde{d}}}   =  |Z|_{ {\mathbf{C}}^{\tilde{d}}},\quad
Z\in \mathbf{C}^{\tilde{d}}.
\end{equation*}
Let $H(B)$ denotes the space of all analytic function in the unit ball,       and let
$\mathbf{Z}_+$ be the set of all non-- negative integers.               Introduce the
operator              $\mathcal{D}_z:  H(B)\rightarrow          H(B)^{\tilde{d}}$  by
\begin{equation*}
\mathcal{D}_z f (z)                         =  (\dots, \partial^\alpha_z f(z),\dots )
\end{equation*}
(the right side contains all                              $\partial^\alpha _z f(z)  =
\partial^{\alpha_1}_{z_1}\cdots \partial ^{\alpha_n}_{z_n} f(z)$ such that        for
$\alpha  =(\alpha_1,\dots,\alpha_d)\in \mathbf{Z}_+^d$ there  holds     $|\alpha|=n$;
therefore it  contains $\tilde{d}$  components). In the Bloch space $\mathcal {B}$ we
will consider   the  following   semi--norm
\begin{equation*}
\|f\|_{\mathcal{B}} =
\sup_{z\in B}   \left(1-|z|^2\right)^n |\mathcal{D}_z f (z)|_{\mathbf{C}^{\tilde{d}}}.
\end{equation*}
Recall that $f\in\mathcal{B}$         if and only     if $\sup_{z\in B}\, (1-|z|^2)^n
\,|\partial^\alpha_z f(z)|$ is finite for all $\alpha\in\mathbf{Z}_+^d,\, |\alpha|=n$,
i.e., if and only if
\begin{equation*}
\sup_{z\in B} \max_{|\alpha|=n}\left(1-|z|^2\right)^n |\partial ^\alpha_z f(z)|
\end{equation*}
is  finite.  It follows from this characterization  of  the Bloch   space  that $f\in
\mathcal{B}$  if and only if $\|f\|_{\mathcal{B}}<\infty$.

Our aim in this paper is to prove the following result.

\begin{theorem}
The  Bergman projection operator  $T_\sigma$ satisfies
\begin{equation*}
\| T_\sigma\|_{L^\infty(B)\rightarrow\mathcal{B}}  \
= \sup_{\|G\|_\infty \le 1}\|T_\sigma  G\|_{\mathcal{B}}
=\mathcal {C} \, \frac{\Gamma(\lambda+n)\Gamma(n)}{\Gamma^2((\lambda +  n)/2)}
\end{equation*}
for every       $\sigma>-1$, where            $\lambda=d+1+\sigma$ and $\mathcal {C}=
\max _{|\zeta|=1}  | Z(\zeta) |_{\mathbf{C}^{\tilde{d}}}$.
\end{theorem}

For       $\zeta\in\mathbf{C}^d$ in the above theorem we write  $Z(\zeta) =
(\dots,\zeta^\alpha,\dots)\in\mathbf{C}^{\tilde{d}}$ (it contains all $\zeta^\alpha =
\zeta_1^{\alpha_1}\cdots \zeta_d^{\alpha_d},\,\alpha\in\mathbf{Z}_+^d,\, |\alpha|=n$).

\begin{remark}\label{RE.C}
Note                    that the constant   $\mathcal {C}$ depends only on  the  norm
$|\cdot|_{\mathbf{C}^{\tilde{d}}}$.                          It is easy to check that
\begin{equation*}
|Z(w)|_{\mathbf{C}^{\tilde{d}}}\le\mathcal{C},\quad w\in B.
\end{equation*}
\end{remark}

\section{Preliminaries}
We will need some auxiliary results in order to prove our theorem.

\subsection{An integral transform}
It       is well known that bi--holomorphic mappings of $B$ onto itself have the form
\begin{equation*}
\varphi_z(\omega)=
\frac {z- {\left<\omega,z\right>} z/{|z|^2}-(1-|z|^2)^{1/2} (\omega- {\left<\omega,z\right>}z/{|z|^2})}
{1-\left<\omega,z\right>}
\end{equation*}
up to  unitary transformations for some $z\in B\backslash \{0\}$.      For  $z=0$ set
$\varphi_0 = -\mathrm {Id}|_B$.
The known  identities
\begin{equation}\label{RE.1}
1-|\varphi_z(\omega)|^2=\frac{(1-|z|^2) (1-|\omega|^2)}{\left|1-\left<\omega,z\right>\right|^2}
\end{equation}
and
\begin{equation}\label{RE.2}
\left(1-\left<\omega,z\right>\right) \left(1-\left<\varphi_z(\omega),z\right>\right)=1-|z|^2
\end{equation}
for  $z,\, \omega\in B$                        will be useful in the proof of the next

\begin{lemma}\label{LE.INT.TRANSFORM}
For every $z\in B$ we have
\begin{equation*}
\left(1-|z|^2\right)^n  \int_{B} \frac {\Phi (w)}{\left|1-\left<z,w\right>\right|^{\lambda+n}}\, dv_\sigma(w)
=\int_{B}  \frac {\Phi(\varphi_z(\omega))}{\left|1-\left<z,\omega\right>\right|^{\lambda-n}}\, dv_\sigma(\omega),
\end{equation*}
where $\Phi(w)$ is a function  in the unit ball $B$ such that the integral on the left
side exists.
\end{lemma}

\begin{proof}
The real Jacobian of $\varphi_z(\omega)$ is given by the expression
\begin{equation*}
(J_\mathbf R\varphi_z)(\omega)=
     \left\{\frac{1-|z|^2}{\left|1-\left<\omega,z\right>\right|^2}\right\}^{d+1}\quad
(\omega\in B).
\end{equation*}
Using~\eqref{RE.1}            we  obtain the next relation for the pull--back measure
\[\begin{split}
dv_\sigma (\varphi_z(\omega))& =
c_\sigma \, (1-\left|\varphi_z(w)\right|^2 )^\sigma   (J_\mathbf R\varphi_z)(\omega)\, dv(\omega)
\\&=c_\sigma \left\{\frac{\left(1-|\omega|^2 \right)\left(1-|z|^2\right)}{\left|1-\left<\omega,z\right>\right|^2}\right\}^\sigma
\left\{\frac{1-|z|^2}{\left|1-\left<\omega,z\right>\right|^2}\right\}^{d+1}dv(\omega)
\\&= \left\{\frac{1-|z|^2}{\left|1-\left<\omega,z\right>\right|^2}\right\}^{\lambda} dv_\sigma(\omega).
\end{split}\]

Denote the  integral on  the left side of our lemma  by $J$. Introducing  the  change
of  variables $w=\varphi_z(\omega)$ and using the preceding   result for          the
pull--back measure, we obtain
\[\begin{split}
J&=\int_B \frac {\left(1-|z|^2\right)^n \Phi (\varphi_z(\omega))} {\left|1-\left<z,\varphi_z(\omega)\right>\right|^{\lambda+n}}
\frac{\left(1-|z|^2\right)^{\lambda}}{\left|1-\left<z,\omega\right>\right|^{2\lambda}}\, dv_\sigma(\omega)
\\&=\int_{B}
\frac {\left(1-|z|^2\right)^{\lambda+n}  \Phi(\varphi_z(\omega))}
{\left|1-\left<z,\varphi_z(\omega)\right>\right|^{\lambda+n} \left|1-\left<z,\omega\right>\right|^{2\lambda}}\, dv_\sigma(\omega)
\\&=\int_{B}\frac{\left(\left|1-\left<z,\omega\right>\right|\left|1-\left<z,\varphi_z(\omega)\right>\right|\right)^{\lambda+n}
\Phi (\varphi_z(\omega))} {\left|1-\left<z,\varphi_z(\omega)\right>\right|^{\lambda+n}
\left|1-\left<z,\omega\right>\right|^{2\lambda}}\, dv_\sigma(\omega)
\\&=\int_{B} \frac {\Phi(\varphi_z(\omega))}{\left|1-\left<z,\omega\right>\right|^{\lambda-n}}\, dv_\sigma(\omega).
\end{split}\]
In the                            third equality we have used the identity~\eqref{RE.2}
\end{proof}

\subsection{Maximum of a parametric  integral}
In  connection with  the next   lemma see~\cite[Proposition~1.4.10]{RUDIN.BOOK.BALL} as
well as~\cite{KALAJ.SCAND}.

\begin{lemma}\label{LE.RUDIN}
For $z\in B$, $c$ real, $t>-1$ define
\begin{equation*}
J_{c,t}(z) = \int_{B}  \frac{\left(1-|w|^2\right)^t}{\left|1-\left<z,w\right>\right|^{d+1+t+c}}\, dv(w).
\end{equation*}
Then $J_{c,t}(z)$ is bounded in $B$ for $c<0$. Moreover, $J_{c,t}(z)$   depends only on
$|z|$, it is  increasing in this value, and
\begin{equation*}
\sup_{z\in B}   J_{c,t}(z) = J_{c,t} (e_1) = \frac{\Gamma(d+1) \Gamma(t+1)  \Gamma(-c)}
{\Gamma^2((d +  1  +  t -c)/ 2)}.
\end{equation*}
\end{lemma}

\subsection{The Vitali theorem}
The       following lemma is known as  Vitali theorem. For example, see    Theorem 26.C
in~\cite{HALMOS.BOOK.MEASURE}.

\begin{lemma}\label{LE.VITALI}
Let        $X$ be a measure space with finite measure $\mu$ and let $h_m : X\rightarrow
\mathbf{C}$ be a  sequence of functions that is uniformly integrable, i.e.,  such  that
for  every $\varepsilon>0$ there  exists $\delta>0$, independent of $m$,     satisfying
\begin{equation*}
\mu(E)<\delta\implies \int_E |h_m|\, d\mu <\varepsilon.                   \eqno(\dag)
\end{equation*}
Now:                             if $\lim_{m\rightarrow \infty} h_m(x)=h(x)$ a.e., then
\begin{equation*}
\lim_{m\rightarrow\infty} \int_X  h_m\, d\mu = \int_X    h\, d\mu.         \eqno(\ddag)
\end{equation*}
In particular, if
\begin{equation*}
\sup_m \int_X  |h_m|^s  d\mu<\infty\ \ \text{for some $s>1$},
\end{equation*}
then  $(\dag)$ and $(\ddag)$ hold.
\end{lemma}

\section{The proof of the main theorem}
We  start now with the  proof of our main theorem. We divide it into   two        parts.
\subsection*{Part I}                   The aim of the first   part is to establish that
\begin{equation*}
\|T_\sigma \|_{ L^\infty(B)\rightarrow \mathcal{B}} \le \mathcal{C}\,
\frac{\Gamma(\lambda+n) \Gamma(n)}{\Gamma^2((\lambda+n)/2)}.
\end{equation*}

We have to precisely estimate from above the expression on the right side of the
expression
\begin{equation*}\begin{split}
\|T_\sigma\|_{ L^\infty(B) \rightarrow \mathcal{B}}\ \ \
 = \sup _{z\in B,\, \|G\|_\infty\le 1} \left(1-|z|^2\right)^n
\left|\mathcal{D}_z(T_\sigma G)(z)\right|_{\mathbf{C}^{\tilde{d}}}.
\end{split}\end{equation*}

Let $z\in B$ be fixed for a moment. For every $Z^*\in\mathbf{C}^{\tilde{d}}$  we    have
\begin{equation*}
\left<\mathcal{D}_z(T_\sigma G)(z),Z^*\right> =
  c_\sigma \int_{B} \left<\mathcal {D}_z {K}_\sigma(z,w),Z^*\right> G(w)\, dv(w).
\end{equation*}
Therefore,
\[\begin{split}
\left|\left<\mathcal{D}_z(T_\sigma  G)(z),Z^*\right>\right|&
\le c_\sigma  \int_{B}  \left|\left< \mathcal{D}_z {K}_\sigma (z,w), Z^*\right>\right|  \left|G(w)\right| dv(w)
\\&\le c_\sigma \int_{B} \left|\left<\mathcal{D}_z {K}_\sigma (z,w), Z^*\right>\right| dv(w)
\end{split}\]
for  every $G\in L^\infty (B)$ which satisfies $\|G\|_\infty\le 1$.  Since   for $\alpha
\in\mathbf{Z}_+^d,\, |\alpha|=n$ we have
\begin{equation*}
D^\alpha _z{K}_\sigma (z,w)
= \frac {\Gamma(\lambda+n)}{\Gamma(\lambda)}\frac{\left(1-|w|^2\right)^\sigma}
{\left(1-\left<z,w\right>\right)^{\lambda+n}}\, \overline{w}^\alpha,
\end{equation*}
we obtain (using Lemma~\ref{LE.INT.TRANSFORM} in the last equality)
\[\begin{split}
\left(1-|z|^2\right)^n \left|\left<\mathcal{D}_z (T_\sigma  G)(z),Z^*\right>\right|
 &\le \frac {\Gamma(\lambda+n)}{\Gamma(\lambda)}\left(1-|z|^2\right)^n
\int_{B}\, \frac{\left| \left<{Z(w)},\overline{Z^*}\right>\right|}{\left|1-\left<z,w\right>\right|^{\lambda+n}}\, dv_\sigma(w)
\\&=\frac {\Gamma(\lambda+n)}{\Gamma(\lambda)}  \int_{B}\,
\frac{\left|\left<{Z(\varphi_z(\omega))}, \overline{Z^*}\right>\right|}{\left|1-\left<z,\omega\right>\right|^{\lambda-n}}\, dv_\sigma(\omega).
\end{split}\]
Since the space $\mathbf {C}^d$ with the norm   $|\cdot|_{ \mathbf{C}^{\tilde{d}}}$   is
(as a finite dimensional space)  reflexive, it follows that
\[\begin{split}
\left(1-|z|^2\right)^n \left|\mathcal{D}_z (T_\sigma G)(z)\right|_{\mathbf{C}^{\tilde{d}}}\
& = \sup_{|Z^*|_{\mathbf{C}^{\tilde{d}}}=1}\left(1-|z|^2\right)^n \left|\left< \mathcal {D}_z (T_\sigma G)(z),Z^*\right>\right|
\\&\le\sup_{|Z^*|_{\mathbf{C}^{\tilde{d}}}=1}\frac {\Gamma(\lambda+n)}{\Gamma(\lambda)}
\int_{B}\frac{\left|\left<{Z(\varphi_z(\omega))},\overline {Z^*}\right>\right|}{\left|1-\left<z,\omega\right>\right|^{\lambda-n}}
\, dv_\sigma(\omega)
\end{split}\]
(in                 the last two expression  we  mean  the   conjugate norm of   $Z^*\in
\mathbf{C}^{\tilde{d}}$).   Now, since for $|Z^*|_{  {\mathbf{C}}^{\tilde{d}}}= 1$ (then
also $|\overline{Z^*}|_{{\mathbf{C}}^{\tilde{d} }}  = 1$,   by our assumption concerning
the norm on $\mathbf{C}^{\tilde{d}}$)   we have
\begin{equation*}\begin{split}
\left|\left<{Z(\varphi_z(\omega))},\overline {Z^*}\right> \right|&\le
|{Z(\varphi_z(\omega))}|_{  {\mathbf{C}}  ^{\tilde{d} }  }\, |\overline{Z^*}|_{  {\mathbf{C}}  ^{\tilde{d} }  }
\le \mathcal{C}
\end{split}\end{equation*}
(recall that  $\mathcal{C}$ stands for the maximum of $|Z(\zeta)|_{  {\mathbf{C}}  ^{\tilde{d} }  }$
on the  unit sphere; see Remark~\ref{RE.C}), we infer
\[\begin{split}
\left(1-|z|^2\right)^n \left| \mathcal{D}_z (T_\sigma G)(z)\right|_{ \mathbf{C}^{\tilde{d}}}
&\le\frac {\Gamma(\lambda+n)}{\Gamma(\lambda)}\, c_\sigma\, \mathcal{C}\,
\int_{B}  \frac {\left(1-|\omega|^2\right)^\sigma}{\left|1-\left<z,\omega\right>\right|^{\lambda-n}}\, dv(\omega)
\\& = \mathcal{C} \, \frac {\Gamma(\lambda+n)}{\Gamma(\lambda)}\, c_\sigma  \, J_{-n,\sigma}(z).
\end{split}\]

By Lemma~\ref{LE.RUDIN} we have
\[\begin{split}
\frac {\Gamma(\lambda+n)}{\Gamma(\lambda)} \,  c_\sigma \, J_{-n,\sigma}(z)
&\le \frac {\Gamma(\lambda+n)}{\Gamma(\lambda)} \, c_\sigma \, J_{-n,\sigma}(e_1)
\\&= \frac {\Gamma(\lambda+n)}{\Gamma(\lambda)} \frac {\Gamma(\lambda)}{\Gamma( \sigma +1)\Gamma(d+1)}
\frac {\Gamma(d+1)\Gamma(\sigma+1) \Gamma(n)}{\Gamma^2( (d+ 1+ \sigma+ n)/2)}
\\&=\frac {\Gamma(\lambda+n) \Gamma(n)}{\Gamma^2(( \lambda  +n)/2)}
\end{split}\]
for all $z\in\ B$.

Finally, we obtain
\begin{equation*}\begin{split}
\|T_\sigma\|_{L^\infty(B)\rightarrow \mathcal{B}}  \ \ \
& = \sup _{z\in B,\, \|G\|_\infty\le 1}
\left(1-|z|^2\right)^n   \left|\mathcal{D}_z(T_\sigma G)(z)\right|_{\mathbf{C}^{\tilde{d}}}
\\&\le \mathcal{C}\, \frac {\Gamma(\lambda+n) \Gamma(n)}{\Gamma^2((\lambda+ n)/2)},
\end{split}\end{equation*}
which is what we wanted to prove.

\subsection*{Part II}
As the second part of the proof we are going to prove             the reverse inequality
\begin{equation*}
\|T_\sigma \|_{ L^\infty(B)\rightarrow \mathcal{B}}\ge
\mathcal{C} \, \frac {\Gamma(\lambda+n) \Gamma(n)}{\Gamma^2((\lambda +  n)/2)}.
\end{equation*}
It is enough to find a  sequence of points  $\{z_m\in B:m\in\mathbf{Z}_+\}$,       and a
sequence of  functions                                        $\{G_m(w)\in L^\infty (B):
 \|G_m \|_\infty \le 1,\, m\in\mathbf{Z}_+\}$ such  that
\begin{equation*}
\liminf_{m\rightarrow\infty}
\left(1-|z_m|^2\right)^n
\left|\mathcal{D}_z ( T_\sigma G_m)(z_m)\right|_{{\mathbf{C}}^{\tilde{d}}}\ge
\mathcal{C}\, \frac {\Gamma(\lambda+n)\Gamma(n)}{\Gamma^2((\lambda + n)/2)}.
\end{equation*}

Recall that we  denoted by $\mathcal {C}$ the value $\sup _{|\zeta|=1}|Z(\zeta)|_{{\mathbf{C}}^{\tilde{d}}}$.
Assume that the maximum is attained for $\zeta_0\in\mathbf{C}^n$.           Thus we have
$|\zeta_0|=1$ and
\begin{equation*}
|Z(\zeta_0)|_{{\mathbf{C}}^{\tilde{d}}} = \mathcal {C}.
\end{equation*}
Since
\begin{equation*}
|Z (\zeta_0)|_{{\mathbf{C}}^{\tilde{d}}} \
 = \sup_{|Z^*|_{{\mathbf{C}}^{\tilde{d}}}=1} \left|\left<Z(\zeta_0),Z^*\right>\right|,
\end{equation*}
we have
\begin{equation*}
|Z(\zeta_0)|_{{\mathbf{C}}^{\tilde{d}}}
= \left|\left<Z(\zeta_0),Z^*_0\right>\right|=\mathcal{C}
\end{equation*}
for  some concrete $|Z^*_0|_{{\mathbf{C}}^{\tilde{d}}}=1$ (we mean the conjugate norm of
$Z^*_0$). Take $\{z_m = \varepsilon_{m}\cdot \zeta_0:m\in\mathbf{Z}_+\}$,          where
$\{\varepsilon_m:m\in\mathbf{Z}_+\}\subseteq (0,1)$    is a sequence of increasing numbers
such that $\lim_{m\rightarrow\infty} \varepsilon_m=1$.  Denote
\begin{equation*}
G_m(w)= \frac{ {\left<{Z(w)},{Z^*_0}\right>}}
{\left|\left<{Z(w)},{Z^*_0}\right>\right|}\frac{\left|1-\left<z_m,w\right>\right|^{\lambda+n}}
{(1-\overline{\left<z_m,w\right>})^{\lambda+n}}\quad (w\in B).
\end{equation*}
Then  $G_m\in  L^\infty (B)$  and $\|G_m\|_\infty=1$.

Since
\begin{equation*}
D^\alpha _z{K}_\sigma (z,w) = \frac {\Gamma(\lambda+n)}{\Gamma(\lambda)}
\frac{ \left(1-|w|^2\right)^\sigma} {\left(1-\left<z,w\right>\right)^{\lambda+n}} \, \overline{w}^\alpha ,
\end{equation*}
it follows that
\begin{equation*}\begin{split}
\left<\mathcal{D}_z{K}_\sigma (z,w),\overline{Z^*_0}\right> &
= \frac {\Gamma(\lambda+n)}{\Gamma(\lambda)}
\frac {\left(1-|w|^2\right)^\sigma}{\left(1-\left<z,w\right>\right)^{\lambda+n}} \, \overline{\left<{Z(w)},{Z^*_0}\right>}.
\end{split}\end{equation*}
Therefore,
\begin{equation*}\begin{split}
&\left| \left<\mathcal {D}_z (T_\sigma G_m)(z_m),\overline{Z^*_0}\right>\right|
\\&= c_\sigma\left|\int_B\, \left<\mathcal{D}_z {K}_\sigma (z_m,w),\overline{Z^*_0}\right>  G_m(w)\, dv(w)\right|
\\& =\frac {\Gamma(\lambda+n)}{\Gamma(\lambda)} \left|\int_B\, \frac{\overline {\left<{Z(w)},{Z^*_0}\right>} } {\left(1-\left<z_m,w\right>\right)^{\lambda+n}}\,  \frac{ {\left<{Z(w)},{Z^*_0}\right>}}
{\left|\left<{Z(w)},{Z^*_0}\right>\right|}\frac{\left|1-\left<z_m,w\right>\right|^{\lambda+n}}
{(1-\overline{\left<z_m,w\right>})^{\lambda+n}}\, dv_\sigma(w)\right|
\\& = \frac {\Gamma(\lambda+n)}{\Gamma(\lambda)}  \int_{B} \frac{ \left| \left<Z(w),{Z^*_0}\right>\right|}
{\left|1-\left<z_m,w\right>\right|^{\lambda+n}}\, dv_\sigma(w).
\end{split}\end{equation*}
Hence
\begin{equation*}\begin{split}
\left(1-|z_m|^2\right)^n  & \left| \left<\mathcal {D}_z(T_\sigma G_m)(z_m),\overline{Z^*_0}\right>\right|
\\&=  \frac {\Gamma(\lambda+n)}{\Gamma(\lambda)} \left(1-|z_m|^2\right)^n
\int_{B} \frac{ \left| \left<Z(w),Z^*_0\right>  \right|}{\left|1-\left<z_m,w\right>\right|^{\lambda+n}}\, dv_\sigma(w)
\\& =  \frac {\Gamma(\lambda+n)}{\Gamma(\lambda)} \int_{B} \frac {\left|\left<Z(\varphi_{z_m}(\omega)),Z^*_0\right>\right|}
{\left|1-\left<z_m,\omega\right>\right|^{\lambda-n}}\, dv_\sigma (\omega),
\end{split}\end{equation*}
where we have used Lemma~\ref{LE.INT.TRANSFORM} in the last equality.

Since
\begin{equation*}
\left|\mathcal {D}_z(T_\sigma G_m)(z_m)\right|_{ \mathbf{C}^{\tilde{d}} }\ =
\sup_{|W^*|_{ \mathbf{C}^{\tilde{d}} }=1} \left| \left<\mathcal{D}(T_\sigma G_m)(z_m),W^*\right>\right|,
\end{equation*}
(we mean the conjugate norm of $W^*\in\mathbf{C}^{\tilde{d}}$), we have
\begin{equation*}\begin{split}
\liminf_{m\rightarrow\infty}\left(1-|z_m|^2\right)^n  & \left|\mathcal{D}_z(T_\sigma G_m)(z_m)\right|_{ \mathbf{C}^{\tilde{d}} }
\\&\ge\liminf_{m\rightarrow\infty}\left(1-|z_m|^2\right)^n  \left| \left<\mathcal{D}_z (T_\sigma G_m)(z_m),\overline{Z^*_0}\right>\right|
\\& = \frac {\Gamma(\lambda+n)}{\Gamma(\lambda)}\liminf_{m\rightarrow\infty}
\int_{B} \frac {\left|\left<Z(\varphi_{z_m}(\omega)),Z^*_0\right>\right|}
{\left|1-\left<z_m,\omega\right>\right|^{\lambda-n}}\, dv_\sigma (\omega).
\end{split}\end{equation*}

We will calculate
\begin{equation*}
\lim_{m\rightarrow\infty} \int_{B}  \frac {\left|\left<Z(\varphi_{z_m}(\omega)),Z^*_0\right>\right|}
{\left|1-\left<z_m,\omega\right>\right|^{\lambda-n}}\, dv_\sigma (\omega).
\end{equation*}
We use the Vitali theorem. For fixed $\omega\in B$ we have
\[\begin{split}
\lim_{m\rightarrow \infty} \frac {\left|\left<Z(\varphi_{z_m}(\omega)),Z^*_0\right>\right|}
{\left|1-\left<z_m,\omega\right>\right|^{\lambda-n}} &
=\frac {\left|\left<Z(\zeta_0),Z^*_0\right>\right|}
{\left|1-\left<\zeta_0,\omega\right>\right|^{\lambda-n}}
=\frac{\mathcal{C}}{{\left|1-\left<\zeta_0,\omega\right>\right|^{\lambda-n}}}.
\end{split}\]
Since for   $s  =  \frac{\lambda- 1/2}{\lambda-n }>1$ (the parameter in Lemma~\ref{LE.VITALI}),
according to Lemma~\ref{LE.RUDIN} (take $c=- 1/2$ and $t=\sigma$), we have
\begin{equation*}
\sup_{m\in\mathbf{Z}_+}\int_{B} \left\{ \frac {\left|\left<Z(\varphi_{z_m}(\omega)),Z^*_0\right>\right|}
{\left|1-\left<z_m,\omega\right>\right|^{\lambda-n}}\right\}^s  dv_\sigma(\omega)
\le \mathcal{C}^s
\sup_{m\in\mathbf{Z}_+} \int_{B} \frac { dv_\sigma(\omega)}{\left|1-\left<\zeta_0,\omega\right>\right|^{\lambda-1/2}}
<\infty
\end{equation*}
(note that
\begin{equation*}
\left|\left<Z(\varphi_{z_m}(\omega)),Z^*_0\right>\right|\le
\left|Z(\varphi_{z_m}(\omega))\right|_{\mathbf{C}^{\tilde{d}}} |Z^*_0|_{\mathbf{C}^{\tilde{d}}}\le \mathcal{C}),
\end{equation*}
by the Vitali theorem we conclude that
\[\begin{split}
\lim_{m\rightarrow \infty}\int_B  \frac {\left|\left<Z(\varphi_{z_m}(\omega)),Z^*_0\right>\right|}
{\left|1-\left<z_m,\omega\right>\right|^{\lambda-n}}\, dv_\sigma(\omega)
&= \int_{B} \lim_{m\rightarrow \infty} \frac {\left|\left<Z(\varphi_{z_m}(\omega)),Z^*_0\right>\right|}
{\left|1-\left<z_m,\omega\right>\right|^{\lambda-n}}\, dv_\sigma(\omega)
\\&=\mathcal{C}\, \int_B \frac{ dv_\sigma (\omega)}{{\left|1-\left<\zeta_0,\omega\right>\right|^{\lambda-n}}}
\\&= \mathcal{C}\, c_\sigma \, J_{-n,\sigma }(\zeta_0).
\end{split}\]

Finally, again by Lemma~\ref{LE.RUDIN} we have
\[\begin{split}
\liminf_{m\rightarrow \infty} \left(1-|z_m|^2\right)^n \left|\mathcal{D}_z(T_\sigma  G_m)(z_m)\right|_{\mathbf{C}^{\tilde{d}}}
&\ge \frac {\Gamma(\lambda+n)}{\Gamma(\lambda)} \, \mathcal{C} \, c_\sigma  \,  J_{-n,\sigma }(e_1)
\\&=\mathcal{C} \, \frac{\Gamma(\lambda + n) \Gamma(n)}{\Gamma^2( (\lambda  +  n)/2)}.
\end{split}\]
This  finishes the second part of the  proof.

\section{Some corollaries}
We  consider here some corollaries of our main result  and  we mention some of earlier
results              concerning the    semi--norms  of the Bergman projection operator.
\subsection{The one--dimensional  case}
The       following corollary       is just a reformulation of our main theorem in the
case $d=1$.

\begin{corollary}    Let $|\cdot| _ {\mathbf{C}}$ be a norm on $\mathbf{C}$.     Let a
semi--norm on the  Bloch space   $\mathcal{B}$   in  the unit disc $\mathbf{U}$ in the
complex plane be given by
\begin{equation*}
\|f\|_\mathcal{B}
= \sup_{z\in\mathbf{U}} \left(1-|z|^2\right)^n |f^{(n)}(z)|_{\mathbf{C}}.
\end{equation*}
Then,  the   Bergman     projection operator $T_\sigma:L^\infty(\mathbf{U})\rightarrow
\mathcal{B}$ satisfies
\begin{equation*}
\| T_\sigma \| _{ L^\infty(\mathbf{U}) \rightarrow\mathcal{B}} =
\mathcal {C} \,     \frac{\Gamma(2+\sigma+n) \Gamma(n)}{\Gamma^2(1 +  (\sigma +  n)/2)}
\end{equation*}
for  every  $\sigma>-1$, where $\mathcal {C} = \max_{|\zeta|=1}  |\zeta|_{\mathbf{C}}$.
\end{corollary}

For $\sigma=0$ and $d=1$ we put $P=T_0$ (then we have the ordinary Bergman projection).
Per\"{a}l\"{a} proved that
\begin{equation*}
\|P\|_{ L^\infty(\mathbf{U})\rightarrow\mathcal{B}} = \frac 8\pi,
\end{equation*}
which is the main results  of~\cite{PERALA.AASF.2012}.

If  the norm in the above corollary is given in the standard way, then $\mathcal{C}=1$,
and therefore,                            for the ordinary Bergman projection  we have
\begin{equation*}
\|P\| _{ L^\infty(\mathbf{U})\rightarrow\mathcal{B}}
 = \frac{ 4 (n+1) \Gamma^2(n)}{n \Gamma^2( n/2)}.
\end{equation*}

\subsection{The   high--dimensional case}
For $Z= (Z_1,\dots, Z_{\tilde{d}})\in\mathbf{C}^{\tilde{d}}$   denote
\begin{equation*}
|Z|_p  = \left\{ |Z_1|^p + \dots + |Z_{\tilde{d}}|^p\right\}^{1/p}\quad (1\le p<\infty)
\end{equation*}
and
\begin{equation*}
|Z|_\infty   = \max_{1 \le j\le \tilde{d} } |Z_j|.
\end{equation*}

In~\cite{KALAJ.SCAND} we find the semi--norm of $T_\sigma$ w.r.t. the semi--norm
on  $\mathcal{B}$ given by
\begin{equation*}
\|f\|_\mathcal{B} =  \sup_{z\in  B}\,  \left(1-|z|^2\right)  \left|\nabla f(z)\right|.
\end{equation*}
In this case 
\begin{equation*}
\|T_\sigma\|_{ L^\infty(B)\rightarrow\mathcal{B} }=  \frac{\Gamma(\lambda+1)}{\Gamma^2((\lambda + 1)/2)}.
\end{equation*}

If a semi--norm on $\mathcal {B}$ is given by
\begin{equation*}\begin{split}
\|f\|_{\mathcal{B}}& = \sup_{z\in B}  \left(1-|z|^2\right)^n |\mathcal{D}_z f(z)|_\infty
= \sup_{z\in B}\, \left(1-|z|^2\right)^n\, \max_{|\alpha|=n}  |\partial^{\alpha}_z f(z)|,
\end{split}\end{equation*}
then    it is not hard to see that $\mathcal{C}=1$. Therefore, by our theorem, we
have
\begin{equation*}
\|T_\sigma\| _{ L^\infty(B)\rightarrow\mathcal{B} }=  \frac{\Gamma(\lambda+n) \Gamma(n)}{\Gamma^2((\lambda+  n)/2)}.
\end{equation*}
This result was obtained in~\cite{KALAJ.JOT}.

We extend the preceding  results  in the following

\begin{corollary}
For $1\le p<\infty$ let a semi--norm on $\mathcal {B}$  be given by
\begin{equation*}\begin{split}
\|f\|_{\mathcal{B}}& = \sup_{z\in B}  \left(1-|z|^2\right)^n |\mathcal{D}_z f(z)|_p
=\sup_{z\in B} \left(1-|z|^2\right)^n \left\{ \sum_{|\alpha| = n }|\partial^\alpha_z f(z)|^p\right\}^{1/p}.
\end{split}\end{equation*}
Then
\begin{equation*}
\|T_\sigma\|_{ L^\infty(B)\rightarrow\mathcal{B} }=
 C_p \, \frac{\Gamma(\lambda+n) \Gamma(n)}{\Gamma^2( (\lambda + n)/2)},
\end{equation*}
where\begin{equation*}
{\tilde{d}} ^{ 1/ p} d^{- n / 2}\le C_p \le {\tilde{d}}^{ 1/p -  1/2},\quad 1\le p<2,
\end{equation*}
and
\begin{equation*}
{C}_p =1,\quad 2\le p<\infty.
\end{equation*}
Particularly, for $n=1$ (then $\tilde{d} =d$) we have
\begin{equation*}
C_p=\left\{
\begin{array}{ll}
d^{ 1/p- 1/2}, & \hbox{if\, $1\le p<2$,}\\
1, & \hbox{if\, $2\le p<\infty$}.
\end{array}
\right.
\end{equation*}
\end{corollary}

This corollary follows from our  main result and the next simple

\begin{lemma}
We have
\begin{equation*}
 {\tilde{d}} ^{ 1/p} d^{-  n/2}\le \max_{|\zeta | =1 } |Z(\zeta)|_p \le{\tilde{d}}^{  1/p-   1/2} ,\quad 1\le p<2,
\end{equation*}
and
\begin{equation*}
\max_{|\zeta | =1 } |Z(\zeta)|_p = 1,\quad 2\le p<\infty.
\end{equation*}
\end{lemma}

\begin{proof}
It is known that for every $W\in \mathbf{C}^{\tilde{d}}$   the sharp inequalities
\begin{equation*}
|W| _p \le \tilde{d}^{ 1/p- 1/2} |W|,\quad 1\le p<2,
\end{equation*}
hold  and
\begin{equation*}
|W| _p \le |W|,\quad 2\le p<\infty.
\end{equation*}

If $|\zeta| =1$, then  $|Z(\zeta)| \le 1$;
\begin{equation*}\begin{split}
|Z(\zeta)|^2  &= \sum _{|\alpha| =n }  |\zeta^\alpha|^2 = \sum_{|\alpha| =n } \{|\zeta_1|^2\}^{\alpha_1}\cdots \{|\zeta_d|^2\}^{\alpha_d}
\\& \le \prod_{j=1}^n (|\zeta_1|^2 + \dots + |\zeta_d|^2)=1.
\end{split}\end{equation*}
We infer that $|Z(\zeta)| _p\le  \tilde{d}^{ 1/p-  1/2} $ if $1\le p<2$, and
$|Z(\zeta)| _p \le1$ if $2\le p<\infty$. On the other hand,  if we take $\zeta =
(d^{- 1/2},\dots,d^{-  1/2})$ for $1\le p< 2$,  and $\zeta=e_1$ in the case
$2\le p<\infty$, we find  $|Z(\zeta)|_p = {\tilde{d}} ^{  1/p} d^{-  n/2}\, (1\le p<2)$,
and  $|Z(\zeta)|_p  = 1\, (2\le p<\infty)$. This proves the lemma.
\end{proof}

\end{document}